\documentclass{article}
\usepackage{amssymb,amsfonts,amsmath,amsthm}
\newtheorem{thm}{Theorem}
\newtheorem{cor}[thm]{Corollary}
\newtheorem{lem}[thm]{Lemma}

\theoremstyle{definition}

\theoremstyle{remark}

\begin{document}

\title{\centering{Boundary Curves of Free Boundary Minimal Surfaces}}
\author{ Zuhuan Yu \thanks{School of Mathematical Science, Capital Normal University, Beijing 100048, China; yuzh@cnu.edu.cn}}





\maketitle

\begin{abstract}
  In this paper we investigate free boundary minimal surfaces in the unit ball in Euclidean 3-space, and by using holomorphic techniques we prove that intersection curves of free boundary minimal surfaces with the unit sphere are all circles.
\end{abstract}



\section{Introduction}

The free boundary minimal surfaces come from the studying of partitioning of convex bodies and it has been studied for a long time. A classical result due to J. C. C. Nitsche [1] is that the minimal disc contained in the unit ball ${\mathbb{B}}^3$ in euclidean space and meeting the boundary $\partial {\mathbb{B}}^3$ orthogonally must be a flat equator, and it has an interesting generalization to the higher codimensions, which was obtained by Fraser and Schoen [2].

In recent years many free boundary minimal surfaces have been constructed out. Fraser and Scheon [3] constructed embedded surfaces of genus zero with any number of boundary components by finding the metrics on these surfaces that maximize the first Steklov eigenvalue with fixed boundary length. Pacard, Folha and Zolotareva [4] found examples of genus zero or one, and with any number of boundary components greater than a large constant. Martin Li and Kapouleas [5] constructed free boundary minimal surfaces with three boundary components and arbitrarily large genus.

There are lots of results of classification of free boundary minimal surfaces in unit ball ${\mathbb{B}}^3$ in euclidean 3-space. Ambrozio and Nunes [6] get a gap theorem for free boundary minimal surfaces which single out the flat disc and critical catenoid. In their paper [7] Fraser and Scheon prove that flat disc is the only free boundary minimal surface with Morse index equals to one and critical catenoid is the only free boundary minimal surface immersed by its first Steklov  eigenfunctions. Smith, Zhou [8], Tran [9] and Devyver [10] state that the Morse index of the critical catenoid equals to four.

In preprint [11], we prove that the free boundary minimal annulus is the critical catenoid, and the methods used there is extended in this paper. We prove that

\begin{thm}{}
 The boundary curves of the free boundary minimal immersed surfaces in the unit ball ${\mathbb{B}}^3$ in euclidean 3-space are all circles.
\end{thm}

\section{The boundary curves are principle lines}

In this section we study the boundary curves of the free boundary minimal surfaces. Let $\Sigma$ be a compact Riemannian surface with boundaries. Assume $$U: \Sigma \rightarrow {\mathbb{B}}^3\subset \mathbb{R}^3$$ is a free boundary minimal surface in the unit ball ${\mathbb{B}}^3$ in Euclidean 3-space $\mathbb{R}^3$, denoted it by $\Sigma$ for brevity.  We prove the following property of the boundary

\begin{lem}
 The connected components of boundary curves of the free boundary minimal surface are all principle lines.
 \end{lem}

\begin{proof}
Take a local arc parameter representation of the boundary curves
$$U:(-\delta, \delta)\rightarrow \mathbb{R}^3;\ \  s\rightarrow U(s).$$

 By the orthogonal assumption, we know that
$U(s)$ and tangent vector $U^\prime(s)$ consist an unit frame in the tangent space $T_{U(s)}\Sigma$. Now let $n(s)$ be the unit normal vector field of the minimal surface $\Sigma$, and $$n(s)=U(s)\times U^\prime(s).$$

Take the derivative of it, we have
$n^\prime(s)=U(s)\times U^{\prime\prime}(s)$. On the other hand $U^{\prime\prime}(s)=a(s)U(s)+b(s)n(s)$, as $|U^{\prime}(s)|=1$. Then
$$n^\prime(s)=U(s)\times U^{\prime\prime}(s)=b(s)U(s)\times n(s)=-b(s) U^\prime(s),$$
which imply $U^\prime(s)$ is a principle direction, namely, the boundary curve $U:(-\delta, \delta)\rightarrow \mathbb{R}^3$ is a principle line on the minimal surface. The Lemma is proved.
\end {proof}

\section{The discrete umbilic points of the minimal surface}
 Let us consider umbilic points on the boundary of the free boundary minimal surface. Besides the plane, the minimal surface has discrete umbilic points on it.

 In fact, if the set of umbilic points have limit points on the minimal surface, then by the holomorphic method, we can prove the Gauss map will be a constant map, so the minimal surface should be flat.

 Apply this result we have

  \begin{lem}
  The umbilic points are discrete on the boundary curves of the free boundary minimal surface except the equator disk.
    \end{lem}

  In the following we only consider the free boundary minimal surfaces which have discrete umbilic points on it.

 \section{The principle patch on the minimal surface}

Let $U:\Sigma \rightarrow \mathbb{R}^3$ be a minimal surface. At any non-umbilic point $P$, there exists a neighborhood which has no umbilic points. We take a principle patch around the point $P$, denoted by
$$ \Omega^{\natural}=\{(x,y)\ |\ -\tilde \delta<x<\tilde \delta, \    -\tilde \delta<y<\tilde \delta.\}$$

On it, the first and second fundamental form of the minimal surface can be written as
$$I=ds^2=Edx^2+Gdy^2,\ \ II=-\kappa Edx^2+\kappa Gdy^2,$$
where $-\kappa, \ \kappa$ are principle curvature with respect to the directions of $x-$lines and $y-$lines.

Note that $\kappa>0$ or $\kappa<0$, corresponding to the choice of the normal direction of the minimal surface.

Without loss of generality, assume that $\kappa>0$. Recall the Codazzi equations
$$\frac{\partial (-\kappa E)}{\partial y}=\frac{(\kappa-\kappa)}{2}\frac{\partial E}{\partial y},\ \
\frac{\partial (\kappa G)}{\partial x}=\frac{(\kappa-\kappa)}{2}\frac{\partial G}{\partial x}.$$
Therefore,
$$\kappa E=\mu(x)^2>0,\ \ \ \kappa G=\nu(y)^2>0.$$
Set $$ u=\int_0^x \mu(x)dx,\ \ v=\int_0^y \nu(y)dy,\ \ \lambda=\frac{1}{\kappa},$$
we have $$\det \frac{\partial (u, v)}{(x,y)}=\mu\nu>0.$$ Choose $(u, v)$ as new parameters we get a new principle patch $\Omega $ (for some $\delta>0$)
\begin{equation}
\Omega =\{(u,v)\ |\ -\delta< u <\delta, \    -\delta< v <\delta \},
\end{equation}
and
\begin{equation}
I=\lambda(du^2+dv^2),\ \ \ II=-du^2+dv^2.
\end{equation}
In summary we have
\begin{lem}
Let $P$ be an non-umbilic point on the minimal surface $U:\Sigma \rightarrow \mathbb{R}^3$, then there exist a principle patch $\Omega $ such that $(u,v)(P)=(0,0)$ and the fundamental forms can be given as in (2).
\end{lem}

\section{The Gauss map of the minimal surface}

Now let $U:\Sigma \rightarrow \mathbb{R}^3$ be a minimal surface, and any non-umbilic point $P$ on it. By the Lemma 4 in last section, there exist a principle patch $\Omega$. While we introduce the complex coordinates $w=u+vi$, then $$\Omega=\{w\in \mathbb{C}\ | \ -\delta<\text{Re}\ w<\delta\,\ -\delta<\text{Im}\ w<\delta\},$$ and the fundamental forms are
$$I=\lambda dw d{\bar w}, \ \ II=-\frac{1}{2}(dw^2+d{\bar w}^2).$$

Firstly, from the fundamental forms we know that the principle curvatures and Gauss curvature are
$$\pm \kappa=\pm \frac{1}{\lambda},\ \ \ K=-\frac{1}{{\lambda}^2}.$$

Secondly,  we have $$K=\frac{1}{\lambda}\Delta (\ln \frac{1}{\sqrt{\lambda}}),$$
where
$$\Delta=\frac{\partial^2}{\partial u^2}+\frac{\partial^2}{\partial v^2}\ \\ =4\frac{\partial^2}{\partial w\partial {\bar w}}.$$
Denote $\phi=\ln \frac{1}{\sqrt \lambda}$, so the Gauss equation is written as
\begin{equation}
\Delta \phi+e^{2\phi}=0.
\end{equation}

The minimal surface has the classical Weierstrass representation, which is a powerful holomorphic tools. The holomorphic data $(g,\omega)$ satisfy
the following conditions

\begin{enumerate}
  \item  $g(w)$ is a meromorphic function, $\omega=f(w)dw $ is a holomorphic 1-form,
  \item  a pole point of order l of $g$ is exactly a zero point of order 2l of $f$,
  \end{enumerate}
The (local) minimal immersion $U: \Omega \rightarrow \mathbb{R}^3$ can be written as bellow
\begin{equation}
U(w)=\text{Re} \int \{(1-g^2)fdw, i(1+g^2)fdw, 2gfdw\}£¬
\end{equation}
and the metric $I=|f(w)|^2(1+|g(w)|^2)^2dwd{\bar w}$, then $\lambda=|f(w)|^2(1+|g(w)|^2)^2, $ hence $\phi=-\frac{1}{2}\ln |f|^2-\frac{1}{2}\ln (1+|g|^2)^2.$ By a direct computation, one shows
\begin{align*}
\Delta \phi =-4\frac{|g_w|^2}{(1+|g|^2)^2},\ \
e^{2\phi} =\frac{1}{\lambda}=\frac{1}{|f|^2(1+|g|^2)^2}.
\end{align*}
Using the Gauss equation (3) we get $2|fg_w|=1,$ thus $$f(w)=\frac{1}{2}\frac{e^{i\theta_0}}{g_w},$$
here $\theta_0$ is a constant real number, and by the condition 2, $g$ only has pole and zero points of order 1. In summary we have
\begin{thm}
The minimal surface is locally determined by its Gauss map and the Weierstrass data can be written as
 \begin{equation}
 (g,\omega)=(g, \frac{1}{2}\frac{e^{i\theta_0}}{g_w}dw).
 \end{equation}
 \end{thm}

 \section{The boundaries of the free boundary minimal surface}


  In this section by using the Weierstrass representation of the minimal surface and the boundary conditions we prove
  \begin{thm}
  The boundary curves of the free boundary minimal surface are planar circles.
  \end{thm}
  \begin{proof} Firstly, if the set of umbilic points on the boundary curves is not discrete, then there exist cluster points, and the free boundary minimal surface will be flat, it is the equator disk. So the claim in the theorem holds.

   Next, consider the set of umbilic points on the boundary is discrete. Take any non umbilic point $P$ on the boundary $\partial \Omega$. In the following we compute the torsion at $P$ of the boundary curves.

   By the results of the up sections, there exists a principle patch around $P$,
   $$U:\Omega \rightarrow \mathbb{R}^3,$$
   where $\Omega=\{w=u+iv \ |\ -\delta<u<\delta,\ \ -\delta<v<\delta\},$ $w(P)=0$, and we have the fundamental forms
   $$I=\lambda(du^2+dv^2)=\lambda dwd{\bar w},$$
   $$ II=-du^2+dv^2=-\frac{1}{2}(dw^2+d{\bar w}^2).$$

   The part of the boundary curves through $P$, contained in $\Omega$, is a principle line, then it is a $u-$line or $v-$line. Without loss of generality, assume it is a $v-$line, determined by $u=0$.

  For simplicity, if needed, by taking a rotation of the ball $\mathbb{B}^3$, we can put $P$ at a special position such that the tangent vector at point $P$ of boundary curve is parallel to $Y$ axis, and the unit normal vector of minimal surface at $P$ is located in $XZ$ plane, properly,
 $$
  \frac{dU(0,v)}{dv}|_{v=0} \parallel (0,1,0),\ \ g(0)=R>1.
 $$
 Here $g(w)$ is holomorphic around the point $w=0$, and we have required the minimal surface can be extended across the boundary defined on a region $\Omega$.

 Now, consider the Weierstrass representation of the minimal surface on the domain $\Omega$. From the theorem, we know the Weierstrass data

$$(g,\omega)=(g(w),f(w)dw),\ \ \ \omega=f(w)dw=\frac{1}{2}\frac{e^{i\theta_0}}{g_w}dw.$$
Since $g(w)$ is holomorphic on the domain $\Omega$ and $g(0)=R$, then $g$ and $g_w$ can be expressed in series of
$$
  \begin{array}{ll}
    g=R+a_1w+a_2w^2+a_3w^3+\cdots , \\
    g_w=a_1+2a_2w+3a_3w^2+\cdots .
  \end{array}
$$
As $\omega$ is holomorphic on the domain $\Omega$ around $w=0$ hence $a_1\neq 0$. Consequently we get
\begin{align*}
\frac{1}{g_w} =&\frac{1}{a_1}-\frac{2a_2}{a_1^2}w+(\frac{4a_2^2}{a_1^3}-\frac{3a_3}{a_1^2})w^2+\\
&(-\frac{8a_2^3}{a_1^4}+\frac{12a_2a_3}{a_1^3}-\frac{4a_4}{a_1^2})w^3+\cdots,\\
\frac{g^2}{g_w} =&\frac{R^2}{a_1}+2R(1-\frac{a_2}{a_1^2}R)w+[(\frac{4a_2^2}{a_1^3}-\frac{3a_3}{a_1^2})R^2-\frac{2a_2}{a_1}R+a_1]w^2+\\
&[(-\frac{8a_2^3}{a_1^4}+\frac{12a_2a_3}{a_1^3}-\frac{4a_4}{a_1^2})R^2+4(\frac{a_2^2}{a_1^2}-\frac{a_3}{a_1})R]w^3+\cdots,\\
\frac{g}{g_w} =&\frac{R}{a_1}+(1-\frac{2a_2}{a_1^2}R)w+[(\frac{4a_2^2}{a_1^3}-\frac{3a_3}{a_1^2})R-\frac{a_2}{a_1}]w^2+\\
&[(-\frac{8a_2^3}{a_1^4}+\frac{12a_2a_3}{a_1^3}-\frac{4a_4}{a_1^2})R+2(\frac{a_2^2}{a_1^2}-\frac{a_3}{a_1})]w^3+\cdots.\\
\end{align*}
Now let $(X,Y,Z)$ denote the coordinates of points in $R^3$, and take the Weierstrass representation of the minimal surface on the domain $\Omega$, $U:\Omega\rightarrow \mathbb{R}^3$
\begin{align*}
 X(w)&=\text {Re}\int^w_{w_0}(1-g(w)^2)\frac{1}{2}\frac{e^{i\theta_0}}{g_w}dw+X_0,\\
Y(w)&=\text {Re}\int^w_{w_0}i(1+g(w)^2)\frac{1}{2}\frac{e^{i\theta_0}}{g_w}dw+Y_0,\\
Z(w)&=\text {Re}\int^w_{w_0}2g(w)\frac{1}{2}\frac{e^{i\theta_0}}{g_w}dw+Z_0.
\end{align*}
So we get
\begin{align*}
 X(w)=&X_0+\frac{1}{2}\text {Re} \ e^{i\theta}\{\frac{1-R^2}{a_1}w-[(1-\frac{a_2}{a_1^2}R)R+\frac{a_2}{a_1^2}]w^2\\
      &+\frac{1}{3}[(\frac{4a_2^2}{a_1^3}-\frac{3a_3}{a_1^2})(1-R^2)+\frac{2a_2}{a_1}R-a_1]w^3+\cdots\},\\
 Y(w)= &Y_0+
\frac{1}{2}\text {Re} \ ie^{i\theta}\{\frac{1+R^2}{a_1}w+[(1-\frac{a_2}{a_1^2}R)R-\frac{a_2}{a_1^2}]w^2\\
       &+\frac{1}{3}[(\frac{4a_2^2}{a_1^3}-\frac{3a_3}{a_1^2})(1+R^2)-\frac{2a_2}{a_1}R+a_1]w^3+\cdots\},\\
 Z(w)= &Z_0+\frac{1}{2}\text {Re} \ 2e^{i\theta}\{\frac{R}{a_1}w+\frac{1}{2}(1-\frac{2a_2}{a_1^2}R)w^2\\
       &+\frac{1}{3}[(\frac{4a_2^2}{a_1^3}-\frac{3a_3}{a_1^2})R-\frac{a_2}{a_1}]w^3+\cdots\}.\\
\end{align*}
On the boundary curve, $w=iv$, appealing the above equations we have the vectors at the point $P(0)$ as
$$\left(
    \begin{array}{c}
      X_v \\
      Y_v \\
      Z_v \\
    \end{array}
  \right)(0)=\frac{1}{4}\left(
                            \begin{array}{c}
                              (1-R^2)(\frac{e^{i\theta_0}}{a_1}-\frac{e^{-i\theta_0}}{\bar {a}_1})i \\
                             -(1+R^2)(\frac{e^{i\theta_0}}{a_1}+\frac{e^{-i\theta_0}}{\bar {a}_1})\\
                              2R(\frac{e^{i\theta_0}}{a_1}-\frac{e^{-i\theta_0}}{\bar {a}_1})i\\
                            \end{array}
                          \right)\parallel \left(
                                             \begin{array}{c}
                                               0 \\
                                               1 \\
                                               0 \\
                                             \end{array}
                                           \right),$$
hence $\frac{e^{i\theta_0}}{a_1}$ is a real number, in other words $a_1=\pm|a_1|e^{i\theta_0},$ and (from boundary assumption at $P(0)$)
$$\left(
    \begin{array}{c}
      X_u \\
      Y_u \\
      Z_u \\
    \end{array}
  \right)(0)=\frac{1}{2}\left(
                            \begin{array}{c}
                              (1-R^2)\frac{e^{i\theta_0}}{a_1} \\
                               0  \\
                              2R\frac{e^{i\theta_0}}{a_1}\\
                            \end{array}
                          \right) \parallel \left(
                                             \begin{array}{c}
                                               X_0 \\
                                               Y_0 \\
                                               Z_0 \\
                                             \end{array}
                                           \right),$$
so the position vector of $P(0)$
$$\left(
                                             \begin{array}{c}
                                               X_0 \\
                                               Y_0 \\
                                               Z_0 \\
                                             \end{array}
                                           \right)=\frac{1}{1+R^2}\left(
                            \begin{array}{c}
                              1-R^2 \\
                               0  \\
                              2R \\
                            \end{array}
                          \right).$$

In the following, for our convenience denoting $\Psi=\frac{4a_2^2}{a_1^2}-\frac{3a_3}{a_1},$ the coordinate functions of the boundary curve near $P$ are
\begin{align*}
X= &X_0+\frac{1}{4}\frac{e^{i\theta_0}}{a_1}\{[(\frac{a_2}{a_1}+\frac{\bar a_2}{\bar a_1})(1-R^2)+R(a_1+\bar a_1)]v^2\\
   &+\frac{i}{3}[(\bar \Psi-\Psi)(1-R^2)+2R(\bar {a}_2-a_2)+a_1^2-\bar {a}_1^2]v^3+\cdots \},\\
Y= &\frac{1}{4}\frac{e^{i\theta_0}}{a_1}\{-2(1+R^2)v+i[(\frac{a_2}{a_1}-\frac{\bar a_2}{\bar a_1})(1+R^2)+R(\bar a_1- a_1)]v^2\\
   &+\frac{1}{3}[(\bar \Psi+\Psi)(1+R^2)-2R(\bar {a}_2+a_2)+a_1^2+\bar {a}_1^2]v^3+\cdots \},\\
Z= &Z_0+\frac{1}{2}\frac{e^{i\theta_0}}{a_1}\{[(a_1+\bar a_1)-2R(\frac{a_2}{a_1}+\frac{\bar a_2}{\bar a_1})]v^2\\
   &+\frac{i}{3}[(\bar \Psi-\Psi)R+a_2-\bar {a}_2]v^3+\cdots \}.\\
\end{align*}

 The normal vector field of the boundary near the point $P$ can be written as
\begin{align*}
X_u= &\frac{1}{4}\frac{e^{i\theta_0}}{a_1}\{2(1-R^2)+2i[(\frac{\bar a_2}{\bar a_1}-\frac{a_2}{a_1})(1-R^2)+R(\bar a_1-a_1)]v\\
     &-[(\bar \Psi+\Psi)(1-R^2)+2R(\bar {a}_2+a_2)-(a_1^2+\bar {a}_1^2)]v^2+\cdots \},\\
Y_u= &\frac{1}{4}\frac{e^{i\theta_0}}{a_1}\{2[(\frac{\bar a_2}{\bar a_1}+\frac{a_2}{a_1})(1+R^2)-R(\bar a_1+a_1)]v\\
     &+i[(\bar \Psi-\Psi)(1+R^2)+2R(a_2-\bar {a}_2)-a_1^2+\bar {a}_1^2]v^2+\cdots \},\\
Z_u= &\frac{1}{2}\frac{e^{i\theta_0}}{a_1}\{2R+i[2(\frac{\bar a_2}{\bar a_1}-\frac{a_2}{a_1})R-\bar a_1+a_1]v\\
      &+[(\bar \Psi+\Psi)R-(\bar {a}_2+a_2)]v^2+\cdots \}.\\
\end{align*}

Next we recall the boundary orthogonal assumption, which is equivalent to $(X,Y,Z)\parallel(X_u,Y_u,Z_u)$, or
\begin{equation}
\frac{X_u}{X} =\frac{Y_u}{Y} =\frac{Z_u}{Z}.
\end{equation}

From the first equality of (6), i.e. $YX_u =XY_u$, comparing the coordinates of $v$ and $v^2$ on sides of the equation we get
\begin{equation}
\frac{e^{i\theta_0}}{a_1}= -\frac{2}{(1+R^2)^2}\{(\frac{\bar a_2}{\bar a_1}+\frac{a_2}{a_1})(1+R^2)-R(\bar a_1+a_1)\},
\end{equation}
\begin{equation}
{\bar \Psi}-\Psi=\Theta_1-\Theta_2
 \end{equation}
$$\Theta_1=\frac{1}{2}\frac{e^{i\theta_0}}{a_1}\{3(\frac{a_2}{a_1}-\frac{\bar a_2}{\bar a_1})(1+R^2)-\frac{R(1+3R^2)}{1-R^2}(\bar a_1-a_1)\},$$
$$ \Theta_2=\frac{1}{1+R^2}\{2R(a_2-\bar {a}_2)-a_1^2+\bar {a}_1^2\}.$$

In the same way, comparing the coordinates of $v$ and $v^2$ on sides of the second equality of (6), i.e. $YZ_u =ZY_u$, we obtain the equation (7) again and
\begin{equation}
{\bar \Psi}-\Psi=\Pi_1-\Pi_2,
\end{equation}
$$\Pi_1=\frac{1}{2}\frac{e^{i\theta_0}}{a_1}\{3(\frac{a_2}{a_1}-\frac{\bar a_2}{\bar a_1})(R^2+1)+(\frac{1}{R}+2R)(\bar a_1-a_1)\},$$
 $$  \Pi_2=\frac{1}{1+R^2}\{2R(a_2-\bar {a}_2)-a_1^2+\bar {a}_1^2\}.$$

On the other hand the boundary curve is on the sphere, so $X^2+Y^2+Z^2=1$. Comparing the coordinates of $v^2$ and $v^3$ on two sides of the equation we also get
\begin{equation}
\frac{e^{i\theta_0}}{a_1}= -\frac{2}{(1+R^2)^3}[(\frac{\bar a_2}{\bar a_1}+\frac{a_2}{a_1})(1+R^4)-R^3(\bar a_1+a_1)],
\end{equation}

\begin{equation}
{\bar \Psi}-\Psi=\Upsilon_1-\Upsilon_2,
\end{equation}
$$\Upsilon_1=\frac{3}{2}\frac{e^{i\theta_0}}{a_1}\{(\frac{a_2}{a_1}-\frac{\bar a_2}{\bar a_1})\frac{(1+R^2)^3}{1+R^4}+\frac{R(1+R^2)^2}{1+R^4}(\bar a_1-a_1)\},$$
  $$\Upsilon_2=\frac{2R^3}{1+R^4}(a_2-\bar {a}_2)+\frac{1-R^2}{1+R^4}(a_1^2-\bar {a}_1^2).$$

To combine the equation (7) with (10) we have
\begin{equation}
\frac{a_2}{a_1}+\frac{\bar a_2}{\bar a_1}=\frac{1}{2R}(a_1+\bar a_1),\ \ \ \frac{e^{i\theta_0}}{a_1}=\frac{R^2-1}{R(1+R^2)^2}(a_1+\bar a_1).
\end{equation}
From (8), (9) and (11) we get that $(1+R^2)^2((a_1^2-\bar a_1^2)=0,$ as  $\frac{e^{i\theta_0}}{a_1}\neq 0$ , then
$$a_1=\bar a_1,\ \  e^{i\theta_0}=2a_1^2\frac{R^2-1}{R(1+R^2)^2}>0,\ \ e^{i\theta_0}=1.$$

Applying it to equations (8,9,11), and with a simple computation we have that
$a_2=\bar a_2,$ and  $ \Psi=\bar \Psi.$  We take the derivatives of coordinates functions with respect to $v$, at point $P$,
$$ X_{vvv}=i\frac{1}{2}\frac{e^{i\theta_0}}{a_1}[(\bar \Psi-\Psi)(1-R^2)+2R(\bar a_2-a_2)+a_1^2-\bar a_1^2]=0,$$
$$Z_{vvv}=i\frac{e^{i\theta_0}}{a_1}[(\bar \Psi-\Psi)R+ a_2-\bar a_2]=0.$$
Hence at point $P$,
$$\det \left(
    \begin{array}{ccc}
      X_v &  X_{vv} &  X_{vvv}\\
      Y_v &  Y_{vv} &  Y_{vvv}\\
     Z_v &  Z_{vv} &  Z_{vvv } \\
    \end{array}
  \right)=\det \left(
    \begin{array}{ccc}
      0 &  X_{vv} &  0\\
      Y_v &  Y_{vv} &  Y_{vvv}\\
     0 &  Z_{vv} &  0 \\
    \end{array}
  \right)=0.$$
Thus the torsion of boundary curve at point $P$ vanishes, i.e.
$$\tau(0)=\frac{\det(U_v,U_{vv},U_{vvv})}{|U_v \times U_{vv}|^2}=0.$$
Since the set of such point $P$ is dense on the boundary curves, thus the boundary curves have no torsion, and which are planar circles.
\end{proof}

\section{The Application of the Theorem}

\subsection{Generalization of Nitsche's Theorem}
\begin{thm}
If the free boundary minimal surface has one boundary component, then it is the equator disk.
\end{thm}
\begin{proof}
By the theorem 6 we know the boundary is a circle, and which must be balanced, so it is the equator. Then the geodesic curvature and the curvature equal to one, so the principle curvature vanishes along the equator. Therefore, all the points on the boundary are umbilic, then the minimal surface is the equator disk.
\end{proof}

In fact from the proof above we have

\begin{cor}
If the free boundary minimal surface has a component of the boundary as an equator, then the minimal surface must be the equator disk.
\end{cor}

\subsection{Principle lines orthogonal to the boundary circles}
\begin{thm}
The boundary point is the geodesic point of the principle lines orthogonal to the boundary circles.
\end{thm}
\begin{proof}It can be proved by a straightforward computation.
\end{proof}

\section*{Acknowledgment}
 This work was partially supported by NSFC (Natural Science Foundation of China) No. 11771070.


\begin{thebibliography}{22}

\bibitem{1 Johannes C. C. Nitsche} Johannes C. C. Nitsche, Stationary partitioning of convex bodies, Arch. Rational Mech. Anal.
89 (1985), no. 1, 1-19. {MR 784101}

\bibitem{2 Ailana Fraser and Richard Schoen} Ailana Fraser and Richard Schoen, Uniqueness theorems for free boundary minimal disks in space forms, Int. Math. Res. Not. IMRN (2015), no. 17, 8268-8274. MR 3404014

\bibitem{3 Ailana Fraser and Richard Schoen} Ailana Fraser and Richard Schoen, The First Steklov eigenvalue, conformal geometry, and
minimal surfaces, Adv. Math. 226 (2011), no. 5, 4011-4030. {MR 2770439}

\bibitem{4 F. Pacard, A. Folha and T. Zolotareva} F. Pacard, A. Folha and T. Zolotareva, Free boundary minimal surfaces in the
unit 3-ball, arXiv: 1502.06812.

\bibitem{5 N. Kapouleas} Nikolaos Kapouleas and Martin Man-chun Li, Free boundary minimal surfaces in the
unit three-ball via desingularization of the critical catenoid and the equatorial disk,
arXiv:1709.08556.

\bibitem{6 Lucas Ambrozio and Ivaldo Nunes}Lucas Ambrozio and Ivaldo Nunes. A gap theorem for free boundary minimal surfaces in the three-ball. arXiv:1608.05689,2016.

\bibitem{7 Ailana Fraser and Richard Schoen} Ailana Fraser and Richard Schoen, Sharp eigenvalue bounds and minimal surfaces in the ball, Invent. Math. 203 (2016), no. 3, 823-890. {MR 3461367}

\bibitem{8 Graham Smith and Detang Zhou} Graham Smith and Detang Zhou, The morse index of the critical catenoid, Geom. Dedicata
201 (2019), no. 1, 13-19.

\bibitem{9 Hung Tran} Hung Tran, Index characterization for free boundary minimal surfaces, to appear in Comm.
Anal. Geom., arXiv:1609.01651.

\bibitem{10 Devyver} Baptiste Devyver, Index of the critical catenoid, arXiv: 1609.02315.


\bibitem{11 yu} Shuangqi Liu, Zuhuan Yu, A note on free boundary minimal annulus, arXiv: 1910.01958.



\end{thebibliography}
\end{document}